\documentclass[11pt,reqno]{article}
\usepackage{amsfonts, amsthm, amsmath, amssymb}
\usepackage[english]{babel}
\theoremstyle{plain}

\newtheorem {theorem} {Theorem}
\newtheorem{proposition}{Proposition}
\theoremstyle{remark}

\DeclareMathOperator{\Dim}{Dim}

\title{A CLT for Plancherel representations of the infinite-dimensional
unitary group }

\author{Alexei Borodin and Alexey Bufetov}

\date{}

\begin{document}

\maketitle

\abstract{We study asymptotics of traces of (noncommutative) monomials
formed by images of certain elements of the universal enveloping algebra of the infinite-dimensional
unitary group in its Plancherel representations. We prove that they converge to
(commutative) moments of a Gaussian process that can be viewed as a collection of
simply yet nontrivially correlated two-dimensional Gaussian Free Fields. The limiting
process has previously arisen via the global scaling limit of spectra for
submatrices of Wigner Hermitian random matrices.

This note is an announcement, proofs will appear elsewhere. }

\section{Introduction}

Asymptotic studies of measures on partitions of representation theoretic origin
is a well-known and popular subject. In addition to its intrinsic importance in
representation theory, see e.g. \cite{BO-European} and references therein,
it enjoys close connections to the theory of random
matrices, interacting particle systems, enumerative combinatorics, and other domains,
for which it often provides crucial technical tools, cf. e.g.
\cite{Ok}, \cite{Olsh-Oxford}.

A typical scenario of how such measures arise is as follows: One starts
with a group with a well known list of irreducible representations, often parametrized
by partitions or related objects. Then a decomposition of a natural reducible representation
of this group on irreducibles provides a split of the total dimension of the
representation space into dimensions of the corresponding isotypical components;
their relative sizes are the weights of the measure. This procedure is well-defined
for finite-dimensional representations, but also for infinite-dimensional representations
with finite trace; the weight of (the label of) an isotypical component is then defined
as the trace of the projection operator onto it, provided that the trace is normalized
to be equal to 1 on the identity operator.

An alternative approach to measures of this sort consists in defining
averages with respect to such a measure for a suitable set of functions on labels
of the irreducible representations. These averages are obtained as traces of the
operators in the ambient representation space that are scalar in each of the
isotypical components. In their turn, the operators are images of central elements
in the group algebra of the group if the group is finite, or in the universal
enveloping algebra of the Lie algebra if one deals with a Lie group. The central
elements form a commutative algebra that is being mapped to the algebra of functions
on the labels, i.e., on partitions or their relatives. The value of the function
corresponding to a central element at a representation label is the (scalar) value
of this element in that representation.

While one may be perfectly satisfied with such an approach from probabilistic point of
view, from representation theoretic point of view it is somewhat unsettling that we
are able to only deal with commutative subalgebras this way, while the main interest
of representation theory is in noncommutative effects.

The goal of this work is go beyond this commutativity constraint.

More exactly, in a specific setting of the finite trace representations of the
infinite-dimensional unitary group described below, we consider a family of commutative
subalgebras of the universal enveloping algebra such that elements from different
subalgebras generally speaking do not commute. We further consider the limit
regime in which the measures for each of the commutative subalgebras are known
to approximate the two-dimensional Gaussian Free Field (GFF), see \cite{BorFer}.
We want to study the ``joint distribution'' of these GFFs for different subalgebras,
whatever this might mean.

For any element of the universal enveloping algebra, one can define its ``average''
as the trace of its image in the representation. Thus, having a representation,
we can define ``averages'' for arbitrary products of elements from our subalgebras,
despite the fact that the elements do not commute.

Our main result is that for certain {\it Plancherel\/} representations, these ``averages''
converge to actual averages of suitable
observables on a Gaussian process that consists of a family of explicitly correlated
GFFs. Thus, the original absence of commutativity in this limit disappears, and yet
the limiting GFFs that arise from different commutative subalgebras do not become
independent.

The same limiting object (the collection of correlated GFFs) has been previously
shown to be the universal global scaling limit for eigenvalues of
various submatrices of Wigner Hermitian random matrices, cf. \cite{Bor}.
We also expect it to arise from other, non-Plancherel factor representations of
the infinite-dimensional unitary group under appropriate limit transitions.

The present paper is an announcement, the proofs will appear in a subsequent publication.

\subsubsection*{Acknowledgements}
The authors are very grateful to Grigori Olshanski for numerous discussions that were
extremely helpful. A.~Borodin was partially supported by NSF grant DMS-1056390.
A.~Bufetov was partially supported by Simons Foundation-IUM scholarship, by Moebius
Foundation for Young Scientists, and by RFBR--CNRS grant 10-01-93114.

\section{Characters of unitary groups}\label{sc:2}
\label{puti}

Let $I$ be a finite set of natural numbers, and let $U(I) = (u_{ij})_{i,j \in I}$
be the group of unitary matrices whose rows and columns are marked by elements of $I$.
In what follows we denote $\{1, 2, \dots, N \}$ as $\overline{1,N}$.
Consider the tower of embedded unitary groups

\begin{equation*}
U(\{ 1 \}) \subset U( \{ 1,2 \}) \subset \dots U(\overline{1,N} ) \subset U(\overline{1,N+1}) \subset \dots,
\end{equation*}
where the embedding $U(\overline{1,k}) \subset U(\overline{1,k+1})$ is defined by
$u_{i,k+1}=u_{k+1,i}=0$, $1 \le i \le k$, $u_{k+1,k+1}=1$.
\emph{The infinite--dimensional unitary group } is the union of these groups:

\begin{equation*}
U(\infty) = \bigcup_{N=1}^{\infty} U( \overline{1,N}).
\end{equation*}

A \emph{signature} (also called \emph{highest weight}) of length $N$ is a sequence of $N$ weakly
decreasing integers $\lambda_1 \ge \lambda_2 \ge \dots \ge \lambda_N$.
Let $\mathbb {GT}_N$ denote the set of such signatures. (Here the letters $\mathbb{GT}$ stand for
`Gelfand-Tsetlin'.)
We say that $\lambda \in \mathbb {GT}_N$ and $\mu \in \mathbb{GT}_{N-1}$ \emph{interlace},
notation $\mu \prec \lambda$, iff $\lambda_i \ge \mu_i \ge \lambda_{i+1}$ for any $1 \le i \le N-1$. We also define $\mathbb{GT}_0$ as a singleton consisting of an element that we denote as
$\varnothing$. We assume that $\varnothing\prec \lambda$ for any $\lambda\in \mathbb{GT}_1$.

The \emph{Gelfand-Tsetlin graph} $\mathbb {GT}$ is defined by specifying its set of vertices as
$\bigcup_{N=0}^{\infty} \mathbb{GT}_N $ and putting an edge between any two signatures $\lambda$ and $\mu$ such that
either $\lambda \prec \mu$ or $\mu \prec \lambda$.
A \emph{path} between signatures $\kappa \in \mathbb {GT}_K$ and $\nu \in \mathbb {GT}_N$, $K<N$ is a sequence
\begin{equation*}
\kappa = \lambda^{(K)} \prec \lambda^{(K+1)} \prec \dots \prec \lambda^{(N)} = \nu, \ \ \  \lambda^{(i)} \in \mathbb{GT}_i.
\end{equation*}
Let $\Dim_N (\nu)$ be the number of paths between $\varnothing$ and $\nu \in \mathbb {GT}_N$. An \emph{infinite path} is a sequence
\begin{equation*}
\varnothing\prec \lambda^{(1)} \prec \lambda^{(2)} \prec \dots \prec \lambda^{(k)} \prec \lambda^{(k+1)} \prec \dots.
\end{equation*}
We denote by $\mathcal P$ the set of all such paths. It is a topological space with
the topology induced from the product topology on the ambient product of discrete sets
$\prod_{N\ge 0}\mathbb{GT}_N$.

For $N=0,1,2,\dots$, let $ M_N$ be a probability measure on $\mathbb {GT}_N$. We say that
$\{ M_N \}_{N=0}^{\infty}$ \emph{is a coherent system of measures} if for any $N\ge 0$
and $\lambda\in\mathbb{GT}_N$,
\begin{equation*}
M_{N} (\lambda) = \sum_{\nu : \lambda \prec \nu} M_{N+1} (\nu)\, \frac{\Dim_{N} (\lambda)}{\Dim_N (\nu)}.
\end{equation*}

Given a coherent system of measures $\{ M_n \}_{n=1}^{\infty}$, define a weight of a
cylindric set of $\mathcal P$ consisting of all paths with prescribed members up to $\mathbb{GT}_N$
by
\begin{equation}
\label{mera-puti}
P ( \lambda^{(1)}, \lambda^{(2)}, \dots, \lambda^{(N)} ) = \frac{ M_N (\lambda^{(N)})}{\Dim_N (\lambda^{(N)} )}.
\end{equation}
Note that this weight depends on $\lambda^{(N)}$ only. The coherency property implies that
these weights are consistent, and they correctly define a Borel probability measure on $\mathcal P$.

It is well known that the irreducible (complex) representations of $U(N)=U(\overline{1,N})$ can be parametrized by signatures of length $N$, and $\Dim_N (\lambda)$ is the dimension of the representation corresponding to $\lambda$. Let $\chi^{\lambda}$ be the conventional character
of this representation (i.e., the function on the group obtained by evaluating trace of the
representation operators) divided by $\Dim_N (\lambda)$.

Define a \emph{character} of the group $U(\infty)$ as a function
$\chi : U(\infty) \to \mathbb C$ that satisfies

1) $\chi(e)=1$, where $e$ is the identity element of $U(\infty)$ (normalization);

2) $\chi(g h g^{-1}) = \chi (h)$, where $g,h$ are any elements of $U(\infty)$ (centrality);

3) $\chi( g_i g_j^{-1})_{i,j=1}^n$ is an Hermitian and positive-definite matrix for
any $n\ge 1$ and $g_1, \dots, g_n \in U(\infty)$ (positive-definiteness);

4) the restriction of $\chi$ to $U(\overline{1,N})$ is a continuous function for any $N\ge 1$
(continuity).

Let $\chi$ be a character of $U(\infty)$. It turns out that for any $N\ge 1$, its restriction to $U(N)$ can be decomposed into a series in $\chi^{\lambda}$,
\begin{equation*}
{\chi |}_{U(N)} = \sum_{\lambda \in \mathbb{GT}_N} M_N (\lambda) \chi^{\lambda},
\end{equation*}
and the coefficients $M_N (\lambda)$ form a coherent system of measures on $\mathbb {GT}$.
Conversely, for any coherent system of measures on $\mathbb{GT}$ one can construct a character
of $U(\infty)$ using the above formula.

The space of characters of $U(\infty)$ is obviously convex. The extreme points of this
set can be viewed as traces of the factor-representations of $U(\infty)$ with finite trace,
or as spherical functions for irreducible spherical unitary representations of the Gelfand pair
$(U(\infty) \times U(\infty), \mathrm{diag} (U(\infty)))$, see \cite{Olsh1} for details.
The classification of the extreme characters is known as Edrei--Voiculescu theorem, see
\cite{OkoOls} and references therein.

\section{Characters and states on the universal enveloping algebra}\label{sc:3}

Let $\mathfrak{gl} (I) = (g_{ij})_{i,j \in I}$ be the complexified
Lie algebra of $U(I)$, let $\mathcal U ( \mathfrak{gl} (I) )$ be its universal enveloping
algebra, and let $Z( \mathfrak{gl} (I) )$ be the center of
$\mathcal U ( \mathfrak{gl} (I) )$.
Denote by
$$
\mathcal U (\mathfrak{gl} (\infty))=
\bigcup_{N\ge 1} \mathcal U ( \mathfrak{gl} (\overline{1,N}) )
$$
the universal enveloping algebra of $\mathfrak{gl} (\infty)$.

There exists a canonical isomorphism
\begin{equation*}
D_I : \mathcal U (\mathfrak{gl} (I) ) \to \mathcal D (I),
\end{equation*}
where $\mathcal D (I)$ is the algebra of left-invariant differential operators on $U(I)$
with complex coefficients. Let $\{ x_{ij} \}$ be the matrix coordinates.
For any character $\chi$ define a \emph{state} on $\mathcal U( \mathfrak{gl} (\infty) )$
as follows: For any
$X \in \mathcal U( \mathfrak{gl} (\infty) )$
\begin{equation}
\label{sost}
\langle X \rangle_{\chi} = D_I (X) \chi (x_{ij}) |_{x_{ij}=\delta_{ij}}, \ \ \ X \in \mathcal U (\mathfrak{gl} (I) ).
\end{equation}
Note that this definition is consistent for different choices of $I$.

We shall denote coordinates of the signatures that parametrize irreducible representations of
$U(I)$ as $\lambda_1^{(I)}, \dots, \lambda_{|I|}^{(I)}$.
There exists a canonical isomorphism  $Z(\mathfrak{gl} (I)) \to \mathbb A(I)$,
where $\mathbb A (I)$ is the algebra of shifted symmetric polynomials in
$\lambda_1^{(I)}, \dots, \lambda_{|I|}^{(I)}$, see e.g. \cite{OkoOlsh}. For any central element,
the value of the corresponding function at a signature corresponds to the (scalar) operator
that this element turns into in the corresponding representation.

Similarly to Section \ref{puti}, restricting $\chi$ to $U(I)$ gives rise to a probability
measure on signatures of length $|I|$. One shows that the state $\langle \,\cdot\,\rangle_{\chi}$
on an element of $Z( \mathfrak{gl} (I))$ equals  the expectation of the corresponding
function in $\mathbb A(I)$ with respect to this probability measure.

One can also evaluate $\langle \,\cdot\,\rangle_{\chi}$ as an expectation on a larger probability
space. Consider $ \{1\}\subset\{1,2\}\subset\ldots \subset \overline{1,k} \subset \dots$.
Let $\mathcal Z$ be the subalgebra in $\mathcal U(\mathfrak{gl} (\infty) )$ generated by
all the centers $Z(\mathfrak{gl} (\overline{1,k}))$, $k=1,2, \dots$. To any element
of $Z(\mathfrak{gl} (\overline{1,k}) )$ we assign a function on the path space $\mathcal P$ by
taking the length $k$ member of the path and applying the isomorphism
$Z(\mathfrak{gl} (\overline{1,k})) \to \mathbb A(\overline{1,k})$. Hence, the
algebra $\mathcal Z$ is naturally embedded into functions on $\mathcal P$.
Denote the probability measure on $\mathcal P$ that arises from the coherent system of measures
originated from $\chi$ by $\mu_{\chi}$.
Then the value of $\langle\, \cdot\, \rangle_{\chi}$ on any element of $\mathcal Z$ is equal to the
expectation of the corresponding function on the probability space $(\mathcal P, \mu_{\chi})$.

\section{One-sided Plancherel characters}

In what follows we restrict ourselves to the
\emph{one--sided Plancherel character} with a growing parameter.
This extreme character is defined as
\begin{equation}
\label{char}
\chi (U) = \exp \left(\gamma L \sum_{i=1}^{\infty} (x_{ii}-1) \right),
\end{equation}
where $U=[x_{ij}]_{i,j\ge 1}$, $\gamma$ is fixed, and $L$ is the growing parameter.
Denote by $\mu_{\gamma}$ the probability measure on $\mathcal P$ that corresponds to this character
(cf. Section \ref{sc:2}),
and denote by $\langle \,\cdot\, \rangle_{\gamma}$ the state that corresponds to this character (cf. Section
\ref{sc:3}).

If one restricts \eqref{char} to $U(I)$ and decomposes it on normalized conventional characters of
$U(I)$, one obtains a probability measure on signatures of length $|I|$ of the form
\begin{equation}
\label{prec}
P^{\gamma L}_I (\lambda) := \begin{cases}
e^{- \gamma L |I|} \dfrac{ (\gamma L)^{\lambda_1 + \dots + \lambda_{|I|}}}{(\lambda_1 + \dots + \lambda_{|I|})!} \dim \lambda
\Dim_{|I|} \lambda, &\lambda_1 \ge \dots \ge \lambda_{|I|} \ge 0; \\
0, &\text{otherwise},
\end{cases}
\end{equation}
where $\dim \lambda$ is the dimension of the irreducible representation of the symmetric group $S_{|\lambda|}$
that corresponds to $\lambda$ (= the number of standard Young tableaux of shape $\lambda$).
Observe that these probability measures are supported by nonnegative signatures, i.e., on partitions
or Young diagrams.

Asymptotic properties of $P_{\overline{1,L}}^{\gamma L}$ as $L \to \infty$
and related distributions have been extensively studied in
\cite{Biane}, \cite{BorFer}, \cite{BorOlsh}, \cite{Meliot}, \cite{Mkr}.

\section{Random height functions and GFF}

A {\it Gaussian family} is a collection of Gaussian random variable $\{ \xi_a \}_{a \in \Upsilon}$
indexed by an arbitrary set $\Upsilon$. We assume that all the random variable are centered, i.e.

\begin{equation*}
\mathbf E \xi_a = 0, \ \ \ \mbox{ for any } a \in \Upsilon.
\end{equation*}
Any Gaussian family gives rise to the \emph{covariance kernel}
$Cov : \Upsilon \times \Upsilon \to \mathbb R$ defined (in the centered case) by
\begin{equation*}
Cov (a_1, a_2) = \mathbf E ( \xi_{a_1} \xi_{a_2} ).
\end{equation*}

Assume that a function $\tilde C : \Upsilon \times \Upsilon \to \mathbb R$
is such that for any $n\ge 1$ and $a_1, \dots, a_n \in \Upsilon$,
$[\tilde C (a_i,a_j)]_{i,j=1}^{n}$ is a symmetric and positive-definite matrix.
Then (see e.g. \cite{Car}) there exists a centered Gaussian family with the covariance
function $\tilde C$.

Let $\mathbb H := \{ z \in \mathbb C : \mathfrak I (z) >0 \}$ be the upper half-plane, and
let $C_0^\infty$ be the space of smooth real--valued compactly supported test functions on
$\mathbb H$. Define a function $C : C_0^\infty \times C_0^\infty \to \mathbb R$ via

\begin{equation*}
C (f_1, f_2) := \int_{\mathbb H} \int_{\mathbb H} f_1 (z) f_2 (w) \left( -\frac{1}{2 \pi} \ln \frac{z-w}{z - \bar w} \right)
dz d \bar z dw d \bar w.
\end{equation*}

The \emph{Gaussian Free Field} (GFF) $\mathfrak G$ on $\mathbb{H}$ with zero boundary conditions can be
defined as a Gaussian family $\{ \xi_f \}_{ f \in C_0^\infty}$ with covariance kernel $C$.
The field $\mathfrak G$ cannot be defined as a random function on $\mathbb H$,
but one can make sense of the integrals $\int f(z) \mathfrak G(z) dz$ over smooth finite contours
in $\mathbb{H}$ with continuous functions $f(z)$, cf. \cite{She}.

Define the \emph{height function}

\begin{equation*}
H : \mathbb R_{\ge 0} \times \mathbb R_{\ge 1} \times \mathcal P \to \mathbb N
\end{equation*}
as
\begin{equation*}
H (x,y) = \sqrt{\pi}\,  \left|\left\{i\in\overline{1,[y]} : \lambda_i^{(y)}  - i + 1/2 \ge x \right\}\right|,
\end{equation*}
where $\lambda_i^{(y)}$ are the coordinates of the signature of length $[y]$ from the
infinite path.
If we equip $\mathcal P$ with a probability measure $\mu_{\gamma}$ then
$H(x,y)$ becomes a random function describing a certain random
stepped surface, or a random lozenge tiling of the half-plane, see \cite{BorFer}.

Define $x(z), y(z) : \mathbb H \to \mathbb R$ via
\begin{equation*}
x(z) = \gamma (1 - 2 \mathfrak R (z)), \ \ y(z) = \gamma |z|^2.
\end{equation*}
Let us carry $H(x,y)$ over to $\mathbb H$ --- define
\begin{equation*}
H^{\Omega} (z) = H ( L x(z), L y(z) ) , \ \ z \in \mathbb H.
\end{equation*}
It is known, cf. \cite{Biane}, \cite{BorFer}, that there exists a limiting (nonrandom) height function
\begin{equation*}
\tilde h (z) := \lim_{L \to \infty} \frac {\mathbf E H^{\Omega} (z) } {L}, \ \ \ z \in \mathbb H,
\end{equation*}
that describes the limit shape. The fluctuations around the limit shape were studied in
\cite{BorFer}, where it was shown that the fluctuation field
\begin{equation}
\mathcal H (z):= H^{\Omega} (z) - \mathbf E H^{\Omega} (z) , \ \ z \in \mathbb H,
\end{equation}
converges to the GFF introduced above.

In \cite{BorFer}, this convergence was proved for a certain space of test functions.
Let us formulate a similar statement that we prove in this work, and that utilizes
a different space of test functions.

Define a moment of the random height function via
\begin{equation*}
M_{y,k} := \int_{-\infty}^{\infty} x^k \bigl(H ( Lx, Ly) - \mathbf E H (Lx, Ly) \bigr) dx.
\end{equation*}
Also define the corresponding moment of the GFF as
\begin{equation*}
\mathcal M_{y,k} = \int_{z \in \mathbb H; y = \gamma |z|^2} x(z)^k \mathfrak G(z) \frac{d x(z)}{dz } dz.
\end{equation*}

\begin{proposition}
As $L \to \infty$, the collection of random variables
$ \{ M_{y,k} \}_{y >0, k \in \mathbb Z_{\ge 0}}$ converges, in the sense
of finite dimensional distributions, to $\{ \mathcal M_{y,k} \}_{y >0, k \in \mathbb Z_{\ge 0}}$.

\end{proposition}

\section{Convergence in the sense of states}
\label{algebra}

Consider a probability space $\Omega$ and a sequence of $k$-dimensional random variables
$( \eta_n^1, \eta_n^2, \dots, \eta_n^k )_{n\ge 1}$ on it that converge, in the sense
of convergence of moments, to a Gaussian random vector $(\eta^1, \dots, \eta^k )$
with zero mean. If we define a \emph{state} as
\begin{equation*}
\langle \xi \rangle_{\Omega} := \mathbf E \xi, \qquad \xi \in L^1 (\Omega),
\end{equation*}
then this convergence can be reformulated as
\begin{multline}
\label{Wick}
\langle \eta_n^{i_1} \eta_n^{i_2} \dots \eta_n^{i_l} \rangle_{\Omega} \xrightarrow[n \to \infty]{}
\sum_{\sigma \in \mathcal{PM}} \prod_{j=1}^{l/2} \langle \eta^{\sigma(2j-1)} \eta^{\sigma(2j)} \rangle_{\Omega},
\\ \mbox{for any $l\ge 1$ and any} \  (i_1, \dots, i_l) \in \{1, \dots, k \}^l,
\end{multline}
where $\mathcal{PM}$ is the set of involutions on $\{ 1,2, \dots, l \}$ with no
fixed points, also known as \emph{perfect matchings}.
(In particular, $\mathcal{PM}$ is empty if $l$ is odd).
Indeed, Wick's formula implies that the right-hand side of \eqref{Wick} contains the
moments of $\eta$.

Let $\mathcal A$ be a $*$-algebra and $ \langle\, \cdot \,\rangle$ be a state
(=linear functional taking non\-ne\-ga\-ti\-ve values at elements of the form $a a^*$) on it.
Let $a_1, a_2, \dots, a_k\in\mathcal A$.

Assume that the elements $a_1, \dots, a_k$ and the state on $\mathcal A$ depend on
a large parameter $L$, and we also have a $*$-algebra $\mathbf A$ generated by elements
$\mathbf a_1, \dots, \mathbf a_k$ and a state $\phi$ on it. We say that
$(a_1, \dots, a_k)$ converge to $(\mathbf a_1, \dots, \mathbf a_k)$
\emph{in the sense of states} if

\begin{equation}
\label{shod}
\langle a_{i_1} a_{i_2} \dots a_{i_l} \rangle \xrightarrow[ L \to \infty]{} \phi( \mathbf a_{i_1} \dots \mathbf a_{i_l} ),
\end{equation}
and this holds for any $l \in \mathbb N$ and any index sets
$(i_1, i_2, \dots, i_l) \in \{1,2, \dots, k \}^l$.

We say that a collection $\{ a_i \}_{i \in \mathfrak J}\subset \mathcal A$ indexed by an arbitrary set
$\mathfrak J$ and depending on a large parameter $L$, converges in the sense of states to a collection
$\{ \mathbf a_i \}_{i \in \mathfrak J}\subset \mathbf A$ if \eqref{shod} holds for any
finite subsets of $\{ a_i \}_{i \in \mathfrak J}$ and corresponding subsets in
 $\{ \mathbf a_i \}_{i \in \mathfrak J}$.

In what follows, the algebra $\mathcal A$ is taken to be
$\mathcal U (\mathfrak{gl} (\infty))$, and the state is taken to be
$\langle \,\cdot\, \rangle_{\gamma}$, cf. \eqref{sost}, \eqref{char}.
The role of the limiting algebra $\mathbf A$ will be played by a commutative
algebra originated from a probability space.

\section{The main result}

Let $A=\{ a_n \}_{n \geq 1}$ be a sequence of pairwise distinct natural numbers.
Let $\mathcal P_A$ be a copy of the path space $\mathcal P$ corresponding to $A$.
Given $A$, we define the height function
\begin{equation*}
H_A : \mathbb R_{\ge 0} \times \mathbb R_{\ge 1} \times \mathcal P_A \to \mathbb N
\end{equation*}
по формуле
\begin{equation*}
H_A (x,y) = \sqrt{\pi}\left| \left\{i\in \overline{1,[y]} : \lambda_i^{ (\{ a_1, \dots, a_{[y]} \}) }  - i + 1/2 \ge x
\right\}\right|,
\end{equation*}
where $\lambda_i^{(\{ a_1, \dots, a_{[y]} \})}$ are the coordinates of
the length $[y]$ signature in the infinite path.
Under the probability measure $\mu_{\gamma}$ on $\mathcal P_A$,
$H_A(x,y)$ becomes a random function on the probability space $(\mathcal P_A, \mu_{\gamma})$.

Let $\{ A_i \}_{i \in \mathfrak J}$ be a family of sequences of pairwise distinct natural numbers,
indexed by a set $\mathfrak J$. Introduce the notation

\begin{equation*}
A_i=\{ a_{i,n} \}_{n \ge 1}, \ \ \
A_{i,m}=\{ a_{i,1}, \dots, a_{i,m} \}.
\end{equation*}

The coordinates $a_{i,j} = a_{i,j}(L)$ may depend on the large parameter $L$.

We say that $\{ A_i \}_{i \in \mathfrak J}$ is \emph{regular} if for any
$i,j \in \mathfrak J$ and any $x,y>0$ there exists a limit

\begin{equation}
\alpha(i,x; j,y) = \lim_{L \to \infty} \frac{|A_{i,[xL]} \cap A_{j,[yL]}|}{L}.
\end{equation}

For example, the following family is regular: $\mathfrak{J} = \{ 1,2,3,4 \}$ with $a_{1,n}=n$,
$a_{2,n}= 2n$, $a_{3,n}=2n+1$, and

\begin{equation*}
a_{4,n} = \begin{cases}
n+L, \ \ \ &n =1,2, \dots, L, \\
n-L, \ \ \ &n=L+1, L+2, \dots, 2L, \\
n, \ \ \ &n \ge 2L+1.
\end{cases}
\end{equation*}

Consider the union of copies of $\mathbb H$ indexed by $\mathfrak J$:
\begin{equation*}
\mathbb H(\mathfrak I): = \bigcup_{i \in \mathfrak I} \mathbb H_i.
\end{equation*}
Define a function
$C: \mathbb H(\mathfrak J) \times \mathbb H(\mathfrak J) \to \mathbb R \cup \{+\infty \} $
via
\begin{equation*}
C_{ij} (z,w) = \frac{1}{2 \pi} \ln \left|\frac{\alpha(i, |z|^2; j, |w|^2) - z w}{\alpha(i, |z|^2; j, |w|^2) - z \bar w}\right|,
\ \ \ \ i,j \in \mathfrak J,\quad z \in \mathbb H_i, w \in \mathbb H_j.
\end{equation*}

\begin{proposition}
For any regular family as above, there exists a generalized Gaussian process on
$\mathbb H( \mathfrak J)$ with the covariance kernel $C_{ij}(z,w)$. More exactly, for any finite
set of test functions $f_m(z) \in C_0^\infty (\mathbb H_{i_m})$ and $i_1, \dots, i_M \in \mathfrak J$,
the covariance matrix
\begin{equation}
cov( f_k,f_l) = \int_{\mathbb H} \int_{\mathbb H} f_k(z) f_l(w) C_{i_k i_l} (z,w) dz d \bar z dw d \bar w
\end{equation}
is positive-definite.
\end{proposition}
\begin{proof} See \cite[Proposition 1]{Bor} .
\end{proof}

Let us denote this Gaussian process as $\mathfrak G_{ \{ A_i \}_{i \in \mathfrak J} }$.
Its restriction to a single half-plane $\mathbb H_i$ is the GFF introduced above because
\begin{equation*}
C_{ii} (z,w) = - \frac{1}{ 2 \pi} \ln \left| \frac{z-w}{z - \bar w} \right|, \qquad z,w \in \mathbb H_i, \quad i\in \mathfrak J.
\end{equation*}
As before, let us carry $H_A(x,y)$ over to $\mathbb H$ --- define
\begin{equation*}
H_A^{\Omega} (z) = H_A ( L x(z), L y(z) ) , \ \ z \in \mathbb H.
\end{equation*}
As was mentioned above, the fluctuations
\begin{equation}
\label{imp}
\mathcal H_i (z) := H_{A_i}^{\Omega} (z) - \mathbf E H_{A_i}^{\Omega} (z) , \qquad i \in \mathfrak J, \ z \in \mathbb H_i,
\end{equation}
for any fixed $i$ converge to the GFF.

The main goal of this paper is to study the \emph{joint} fluctuations \eqref{imp} for different
$i$. The joint fluctuations of $\mathcal H_i$ are understood as follows.
Define the moments of the random height function as
\begin{equation}
\label{moment}
M_{i,y,k} := \int_{-\infty}^{\infty} x^k \bigl(H_{A_i}( Lx, Ly) - \mathbf E H_{A_i} (Lx, Ly) \bigr) dx.
\end{equation}

It turns out that the function $M_{i,y,k}$ belongs to $\mathbb A (A_{i, [Ly]})$, and thus it corresponds
to an element of $Z(\mathfrak {gl} (A_{i, [Ly]} ))$; denote this element by the same symbol.
Note that all such elements  $M_{i,y,k}$ for all $i,y,k$ belong to the ambient algebra
$\mathcal U (\mathfrak{gl} (\infty))$, and we also have the state $\langle \,\cdot\, \rangle_{\gamma}$ defined
on this ambient algebra. Thus, we can talk about convergence of such elements in the sense of states,
see Section \ref{algebra}.
We are interested in the limit as $L \to \infty$.

We prove that the family $\{ \mathcal H_i \}_{i \in \mathfrak{J} }$ converges to the generalized
Gaussian process $\mathfrak G_{ \{ A_i \} _{ i \in \mathfrak J } }$.
Define the moments of $\mathfrak G_{ \{ A_i \}_{i \in \mathfrak J}}$ by
\begin{equation*}
\mathcal M_{i,y,k} = \int_{z \in \mathbb H; y = \gamma |z|^2} x(z)^k \mathfrak G_{A_i} (z) \frac{d x(z)}{dz } dz.
\end{equation*}

\begin{theorem}\label{th:1}
As $L \to \infty$, the moments
$\{M_{i,y,k}\}_{i \in \mathfrak J, y>0, k \in \mathbb Z_{\ge 0}}$ converge in the sense of states to the
moments $\{ \mathcal M_{i,y,k} \}_{i \in \mathfrak J, y>0, k \in \mathbb Z_{\ge 0}}$.
\end{theorem}

Thus, in the $L\to\infty$ limit, the noncommutativity disappears
(limiting algebra $\mathbf A$ is commutative), and yet the random fields $\mathcal H_i$ for different
$i$'s are not independent.

Let $u=L x$. The definition of the height function implies
\begin{equation*}
\frac{d}{d u} H_{A_i} (u, [Ly]) = - \sqrt{\pi}\, \sum_{s=1}^{[Ly]} \delta\left( u - \left(\lambda_s^{(A_{i,[Ly]})} -s + 1/2\right)\right).
\end{equation*}

Define the \emph{shifted power sums}
\begin{equation*}
p_{k,I} = \sum_{i=1}^{|I|} \left(\left(\lambda_i^{(I)} - i +\frac12\right)^k - \left(-i +\frac12\right)^k\right) , \ \ \ \ I \subset \mathbb N.
\end{equation*}
One shows that $p_{k,I} \in \mathbb A(I)$, and hence they correspond to
certain elements of $Z(\mathfrak {gl} (I))$ that we will denote by the same symbol.

Integrating \eqref{moment} by parts shows that $M_{i,y,k}$ can be rewritten as
\begin{multline*}
\frac{L^{-(k+1)} \sqrt{\pi}}{ k+1} \left( \sum_{s=1}^{[Ly]} \left(\lambda_s^{(A_{i,[Ly]})} -s + \frac 12\right)^{k+1}
- \mathbf E \sum_{s=1}^{[Ly]} \left(\lambda_s^{(A_{i,[Ly]})} -s + \frac 12\right)^{k+1} \right) \\ =\frac{L^{-(k+1)} \sqrt{\pi}}{ k+1}
\,\left( p_{k+1,I} - \mathbf E p_{k+1,I}\right).
\end{multline*}
Thus, Theorem \ref{th:1} can be reformulated as follows.

\begin{theorem}
Let $k_1, \dots, k_m \ge 1$ and $I_1, \dots, I_m$ be finite subsets of $\mathbb N$ that
may depend on the large parameter $L$ in such a way that there exist limits
\begin{equation*}
\eta_r=\lim_{L \to \infty} \frac{|I_r|}{L} > 0, \ \ \ \ c_{rs} = \lim_{L \to \infty} \frac{|I_r \cap I_s|}{L}\,.
\end{equation*}
Then as $L \to \infty$, the collection
\begin{equation*}
\left(L^{-k_r}\left( p_{k_r, I_r}- \mathbf E p_{k_r, I_r}\right)\right)_{r=1}^{m}
\end{equation*}
 of elements of\, $\mathcal U (\mathfrak{gl} (\infty))$ converges in the sense of states, cf. \eqref{shod}, to the Gaussian vector
$(\xi_1, \dots, \xi_m)$ with zero mean and covariance
\begin{multline*}
\mathbf {E} \xi_r \xi_s =
\frac{ k_r k_s}{\pi} \oint_{|z|^2 = \frac{\eta_r}{\gamma}; \mathfrak I (z)>0} \oint_{|w|^2 = \frac{\eta_s}{\gamma} ;
\mathfrak I(w)>0}
(x(z))^{k_r-1} (x(w))^{k_s-1} \\ \times
\frac{1}{2 \pi} \ln \left|\frac{c_{rs}/ \gamma - z w}{c_{rs}/ \gamma - z \bar w}\right| \frac{d(x(z))}{dz}
\frac{d(x(w))}{dw} dz dw.
\end{multline*}

\end{theorem}

\end{document}